\documentclass{amsart}

\usepackage{amsmath,amssymb,enumerate}
\usepackage{mathptmx}

\usepackage{microtype}
\usepackage{colonequals}

\usepackage{hyperref}

\usepackage[nobysame]{amsrefs}

\usepackage[all]{xy}
\SelectTips{cm}{}

\theoremstyle{plain}

\newtheorem{theorem}{Theorem}[section]
\newtheorem{proposition}[theorem]{Proposition}
\newtheorem{lemma}[theorem]{Lemma}
\newtheorem{corollary}[theorem]{Corollary}

\theoremstyle{definition}

\newtheorem{example}[theorem]{Example}

\theoremstyle{remark}
\newtheorem{remark}[theorem]{Remark}
\newtheorem*{claim}{Claim}

\newtheorem{chunk}[theorem]{}

\numberwithin{equation}{section}

\newcommand{\bbC}{\mathbb{C}}
\newcommand{\bbZ}{\mathbb{Z}}

\newcommand{\ulx}{\underline{x}}

\newcommand{\mcV}{\mathcal{V}}

\newcommand{\fm}{\mathfrak{m}}
\newcommand{\fn}{\mathfrak{n}}
\newcommand{\fp}{\mathfrak{p}}
\newcommand{\fq}{\mathfrak{q}}

\newcommand{\sfh}{\mathsf{h}}

\newcommand{\susp}{\mathsf{\Sigma}}

\newcommand{\dsing}{\mathsf{D}_{\rm sg}}

\newcommand{\dbcat}{\mathsf{D}^{\mathsf b}}

\newcommand{\ges}{\geqslant}

\newcommand{\lla}{\longleftarrow}
\newcommand{\lra}{\longrightarrow}

\newcommand{\xla}{\xleftarrow}
\newcommand{\xra}{\xrightarrow}

\newcommand{\ann}{{\operatorname{ann}\,}}

\newcommand{\qdiff}[3]{{\operatorname{Q}^{#1}({#3}/{#2})}}
\newcommand{\depth}{{\operatorname{depth}\,}}
\renewcommand{\dim}{{\operatorname{dim}\,}}

\newcommand{\edim}{{\operatorname{emb\,dim}\,}}

\newcommand{\Ext}{\operatorname{Ext}}
\newcommand{\fitt}{\operatorname{Fitt}}
\newcommand{\gldim}{{\operatorname{gldim}\,}}
\newcommand{\height}{{\operatorname{height}\,}}

\newcommand{\HCH}[3]{\operatorname{HH}^{#1}(#3/#2)}
\newcommand{\HH}[2]{\operatorname{H}^{#1}(#2)}

\newcommand{\jac}{{\operatorname{jac}\,}}
\newcommand{\kdiff}[2]{{\operatorname{K}({#2}/{#1})}}
\newcommand{\lotimes}[1]{\otimes^{\mathbf{L}}_{{#1}}}
\renewcommand{\mod}{{\operatorname{mod}\,}}
\newcommand{\Mod}{{\operatorname{Mod}\,}}

\newcommand{\Spec}{{\operatorname{Spec}\,}}

\newcommand{\thick}{\operatorname{thick}}
\newcommand{\Tor}{\operatorname{Tor}}

\newcommand{\cent}[1]{{#1}^{\!\mathsf c}}
\newcommand{\env}[1]{{#1}^{\!\mathsf e}}
\newcommand{\denv}[1]{{#1}^{\!(\mathsf e)}}
\newcommand{\ndiff}[2]{{\operatorname{N}({#2}/{#1})}}
\newcommand{\opp}[1]{{ {#1}^{\!\mathsf o}}}

\newcommand{\Ker}{\operatorname{Ker}}
\newcommand{\Hom}{\operatorname{Hom}}
\newcommand{\RHom}{\operatorname{\mathbf{R}Hom}}

\begin{document}

\title[Jacobian ideal]{The Jacobian ideal of a commutative ring and \\ annihilators of cohomology}

\author{Srikanth B. Iyengar}
\address{Department of Mathematics, University of Utah, Salt Lake City, UT 84112-0090, USA}
\email{iyengar@math.utah.edu}

\author{Ryo Takahashi}
\address{Graduate School of Mathematics, Nagoya University, Furocho, Chikusaku, Nagoya 464-8602, Japan}
\email{takahashi@math.nagoya-u.ac.jp}

\date{28th July 2018}


\begin{abstract}
It is proved that for a ring $R$ that is either an affine algebra over a field, or an equicharacteristic complete local ring, some power of the Jacobian ideal of $R$ annihilates $\mathrm{Ext}^{d+1}_{R}(-,-)$, where $d$ is the Krull dimension of $R$. Sufficient conditions are identified under which the Jacobian ideal itself annihilates these Ext-modules, and examples are provided that show that this is not always the case. The crucial new idea is to consider a derived version of the Noether different of an algebra.
\end{abstract}

\maketitle

\section{Introduction}
Consider a commutative noetherian ring $R$ of Krull dimension $d$. One characterization of the property that $R$ is regular is that $\Ext^{d+1}_{R}(M,N)=0$ for all $R$-modules $M$ and $N$. A natural measure then of the failure of a ring $R$ to be regular is the ideal of the elements of $R$ that annihilate these Ext-modules. If this ideal contains a element that is not a zerodivisor, then the regular locus of $R$ contains a nonempty open subset of $\Spec R$. Thus, it can happen that this ideal is zero, even when $R$ is a domain.

Our first result, contained in Theorem~\ref{th:jac-djac}, is that for rings of a geometric origin, some fixed power of elements in $\jac(R)$, the Jacobian ideal of $R$, annihilate $\Ext^{d+1}_{R}(-,-)$.

\begin{theorem}
\label{th:main1}
Let $R$ be either an affine algebra over a field, or  an equicharacteristic complete local ring, of Krull dimension  $d$. There exists an integer $s$ such that
\[
\jac(R)^{s} \cdot \Ext^{d+1}_{R}(M,N) =0 
\]
for all  $R$-modules $M,N$.
\end{theorem}

For affine domains over a perfect field (this includes the characteristic zero case), the result above was proved by Wang~\cite{Wang99}*{Theorem~3.7} using rather different arguments, and building on his earlier work that treats the case of equicharacteristic complete local rings  that are equidimensional and with a perfect residue field; see \cite{Wang94}*{Theorem 5.4}.

In \cite{Wang94}*{Question 2} Wang asks if for any $d$-dimensional complete local ring $R$ containing a field, its Jacobian ideal annihilates $\Ext^{d+1}_{R}(-,-)$. Said otherwise, does $s=1$ suffice in Theorem~\ref{th:main1}?  The result below provides a partial answer to this question. 

\begin{theorem}
\label{th:main2}
Let $R$ be either an affine algebra over a field, or  an equicharacteristic complete local ring, of Krull dimension  $d$.  If $R$ is equidimensional and $ 2\, \depth R_{\fp}\ge \dim R_{\fp}$ for each $\fp$ in $\Spec R$, then for all $R$-modules $M,N$, one has
\[
\jac(R) \cdot \Ext^{d+1}_{R}(M,N) =0\,.
\]
\end{theorem}

In Section~\ref{se:examples} we provide examples that show that the conclusion of the theorem above does not hold in general, and suggest that the hypotheses we impose are probably optimal. The hypotheses of the preceding theorem are satisfied when $R$ is a Cohen-Macaulay ring and then one gets that $\jac(R)$  annihilates $\Ext^{d+1}_{R}(M,N)$ for any $R$-modules $M,N$. In this way one recovers \cite{Wang94}*{Theorem~5.3}.

The Jacobian ideal of an affine algebra, or of an equicharacteristic complete local ring, can be realized as the sum of K\"ahler differents of $R$ over its various Noether normalizations. The proofs of the predecessor of Theorems~\ref{th:main1} and \ref{th:main2} have all exploited this fact, and so  do we. The link to annihilators of Ext-modules is usually made via Noether differents, which contain the K\"ahler differents and coincide with them up to radical. For Cohen-Macaulay rings it is well-known, and not difficult to prove, that any such Noether different annihilates the appropriate Ext-modules. This is not the case in general, which points to the difficulty in dealing with rings that are not Cohen-Macaulay. 

The new idea in our work is to consider a derived analogue of the Noether different,  introduced in Section~\ref{se:de-ndiff}. It is a routine computation to check that these annihilate the Ext-modules. We prove they are contained in the Noether different, agree with it up to radical, and coincide with it under certain conditions on the ring $R$ that are less stringent than the Cohen-Macaulay property; this is what leads to Theorems~\ref{th:main1} and \ref{th:main2}.

To wrap up the Introduction, we give a few reasons we care about the results presented here. To begin with, there is a close connection between the existence of nonzero elements of a ring $R$ that annihilate its Ext-modules and generators for the derived category of $R$-modules. This relationship is explored in \cites{IyengarTakahashi14a, IyengarTakahashi16a}, and  led us to the work reported in this paper. The annihilators of Ext-modules also give information on the Fitting invariants of syzygies of finitely generated modules, as the title of \cite{Wang94} indicates; see in particular, Proposition 2.4 and Theorems 5.1 and 5.2 of \emph{op.~cit}.  Finally, one might view Theorems~\ref{th:main1} and \ref{th:main2} as quantitive enhancements of the classical Jacobian criterion for detecting smoothness of affine algebras; see Corollary~\ref{co:jac-djac} and Remark~\ref{re:singular-locus}.

\section{Derived Noether different}
\label{se:de-ndiff}

In this section we introduce a notion of a derived Noether different of an algebra and relate it to the classical Noether different.
With an eye on the future, the construction is described for general associative algebras. The principal results are Theorems~\ref{th:qdiff-ann} and \ref{th:separable-flat}.

In what follows, given a ring $\Lambda$, we write $\Mod \Lambda$ (respectively, $\mod \Lambda$) for the category of (finitely presented) $\Lambda$-modules. By a `module' we mean a `left' module, unless specified otherwise. 

Let $A$ be a commutative ring and $\Lambda$ an $A$-algebra. This means that there is a homomorphism of rings $A\to \Lambda$ with image in $\cent \Lambda$, the center of $\Lambda$. We begin by recalling the construction of the Noether different; see~\cite{AuslanderGoldman60}.

\subsection*{Noether different}
We write $\opp \Lambda$ for the opposite algebra of $\Lambda$ and set
\[
\env \Lambda \colonequals\  \Lambda \otimes_{A}\opp \Lambda\,.
\]
This is the enveloping algebra of the $A$-algebra $\Lambda$. Modules over $\env\Lambda$ are precisely the left-right $\Lambda$ bimodules; indeed, given such a bimodule $M$, the action of $\env\Lambda$ is given by
\[
(\lambda\otimes \lambda')m \colonequals\  \lambda m\lambda' \,.
\]
In particular, $\Lambda$ itself is a module over $\env\Lambda$ and the natural multiplication map
\[
\mu\colon \env\Lambda \lra \Lambda\quad\text{defined by $\mu(\lambda \otimes \lambda')=\lambda \lambda'$}
\]
is one of $\env\Lambda$-modules. This induces a map 
\[
\Hom_{\env \Lambda}(\Lambda,\mu) \colon 
\Hom_{\env \Lambda}(\Lambda,\env\Lambda)  \lra \Hom_{\env \Lambda}(\Lambda,\Lambda)=\cent\Lambda\,.
\]
The image of this map is the \emph{Noether different} of the $A$-algebra $\Lambda$, that we denote $\ndiff A\Lambda$. Thus, an element $z \in \cent\Lambda$ is in $\ndiff A\Lambda$ precisely when the map $\Lambda\xra{\ z\ }\Lambda$ of $\env\Lambda$-modules factors through the map $\mu$; said otherwise, there is a commutative diagram of $\env\Lambda$-modules
\[
\xymatrixrowsep{.5cm}
\xymatrixcolsep{.5cm}
\xymatrix{
\Lambda\ar@{->}[rr]^{z}\ar@{-->}[dr]&& \Lambda\\
& \env \Lambda \ar@{->}[ur]_{\mu}&
}
\]

In what follows, we need a derived version of the Noether different. Its definition is based on derived Hochschild cohomology functors. For details of the construction of the latter gadget, which requires the use of DG (=Differential Graded) algebras and modules, we refer the reader to \cite{AvramovIyengarLipmanNayak10}*{\S3}.

\subsection*{Derived Noether different}
Consider the derived enveloping algebra of the $A$-algebra $\Lambda$:
\[
\denv \Lambda \colonequals\  \Lambda \lotimes A\opp \Lambda\,.
\]
This is a DG $A$-algebra, realized as $F\otimes_{A}\opp F$, where $F$ is a flat DG algebra resolution of the $A$-algebra $\Lambda$. It comes equipped with a morphism of  DG algebras
\[
\sfh \colon \denv\Lambda \lra \HH 0{\denv \Lambda} = \env \Lambda\,,
\]
where $\env\Lambda$ is viewed as a DG algebra concentrated in degree zero. Any $\env\Lambda$-module, and in particular $\Lambda$, has an induced structure of a DG module over $\denv\Lambda$. For each integer $n$ and DG $\denv\Lambda$-module $X$, 
 the $n$th \emph{derived Hochschild cohomology} of the $A$-algebra $\Lambda$ with coefficients in $X$ is the $A$-module
\[
\Ext^{n}_{\denv \Lambda}(\Lambda,X)\,.
\]
Mimicking the construction of the Noether different, we introduce the \emph{derived Noether different} of the $A$-algebra $\Lambda$ as the graded $A$-module
\[
\qdiff *A\Lambda\colonequals\  \mathrm{Image}(\Ext^{*}_{\denv \Lambda}(\Lambda,\denv\Lambda) \lra \Ext^{*}_{\denv \Lambda}(\Lambda,\Lambda))
\]
where the map is the one induced by the composition of morphisms $\denv\Lambda\xra{\sfh}\env \Lambda\xra{\mu}\Lambda$. 
We chose the letter `Q' to denote this different because derived Hochschild cohomology was introduced by Quillen~\cite{Quillen:1968a}.

Observe that $\Ext^{0}_{\denv \Lambda}(\Lambda,\Lambda)$ is naturally isomorphic to $\Hom_{\env \Lambda}(\Lambda,\Lambda)$, that is to say, to $\cent \Lambda$, so $\qdiff 0A\Lambda$ is  in the center of $\Lambda$. The result below relates this to the Noether different of the $A$-algebra $\Lambda$. The hypothesis on $\Tor^{A}_{*}(\Lambda,\Lambda)$ holds when $\Lambda$ is flat as an $A$-module, but not only; see Theorem~\ref{th:half-CM}.
 
\begin{lemma}
\label{le:qdiff-ndiff}
There is an inclusion $\qdiff 0A{\Lambda}\subseteq\ndiff A\Lambda$; equality holds if $\Tor^{A}_{i}(\Lambda,\Lambda)=0$ for $i\ge 1$.
\end{lemma}

\begin{proof}
The inclusion is justified by the natural factorization 
\[
\xymatrixcolsep{3pc}
\xymatrixrowsep{1.5pc}
\xymatrix{
\Ext^{0}_{\denv \Lambda}(\Lambda,\denv\Lambda) \ar@{->}[r]^{\Ext^{0}_{\denv\Lambda}(\Lambda,\sfh)}
				 \ar@{->}[dr]_{\Ext^{0}_{\denv\Lambda}(\Lambda,\mu\sfh)}
	& \Ext^{0}_{\denv \Lambda}(\Lambda,\env\Lambda) \ar@{->}^{\sim}[r] 
	& \Ext^{0}_{\env \Lambda}(\Lambda,\env\Lambda) \ar@{->}[dl]^{\Ext^{0}_{\env\Lambda}(\Lambda,\mu)} \\
	&\Ext^{0}_{\env\Lambda}(\Lambda,\Lambda)=\cent\Lambda}
\]
When $\Tor^{A}_{i}(\Lambda,\Lambda)=0$ for $i\ge 1$, the map $\denv\Lambda\to\env \Lambda$ is a quasi-isomorphism, so the first horizontal map is also an isomorphism and hence the two differents coincide.
\end{proof}

\subsection*{Products}
Given any DG $A$-algebra $B$ and DG $B$-modules $M,N$, we view the elements in $\Ext^{n}_{B}(M,N)$ as morphisms $M\to \susp^{n}N$ in the derived category of DG $B$-modules. Composition makes the graded $A$-module $\Ext^{*}_{B}(M,M)$ a graded $A$-algebra, and $\Ext^{*}_{B}(M,N)$ a graded left $\Ext^{*}_{B}(N,N)$ and right $\Ext^{*}_{B}(M,M)$ bimodule over it. These actions are compatible with morphisms; for example, if $N\to N'$ is a morphism of DG $B$-modules, the induced map
\[
\Ext^{*}_{B}(M,N)\lra \Ext^{*}_{B}(M,N')
\]
is one of right $\Ext^{*}_{B}(M,M)$-modules.

Returning to our context: The graded $A$-algebra $\Ext^{*}_{\denv\Lambda}(\Lambda,\Lambda)$ is graded-commutative. Since $\Ext^{*}_{\denv \Lambda}(\Lambda,\mu\sfh)$ is linear with respect to the action of $\Ext^{*}_{\denv\Lambda}(\Lambda,\Lambda)$, it follows that $\qdiff *A\Lambda$ is an ideal in the ring $\Ext^{*}_{\denv\Lambda}(\Lambda,\Lambda)$.

Akin to the description of elements in $\ndiff A\Lambda$, an element $\alpha$ in $\Ext^{s}_{\denv\Lambda}(\Lambda,\Lambda)$ is in $\qdiff sA\Lambda$ precisely when there is a factorization
\begin{equation}
\label{eq:factorization}
\xymatrix{
\susp^{-s}\Lambda \ar@{->}[rr]^{\alpha}\ar@{-->}[dr]_{\eta}&& \Lambda \\
&\denv \Lambda \ar@{->}[ur]_{\mu\sfh}&
}
\end{equation}
in the derived category of DG $\denv\Lambda$-modules.

\subsection*{Annihilators of Ext}
Next we consider the action of $\Ext^{*}_{\denv\Lambda}(\Lambda,\Lambda)$ on $\Ext^{*}_{\Lambda}(M,N)$, for any complexes of $\Lambda$-modules $M,N$. This action is realized through the homomorphism
\[
\Ext^{*}_{\denv\Lambda}(\Lambda,\Lambda)\lra \Ext^{*}_{\Lambda}(M,M)
\]
of graded rings, where a morphism $\alpha\colon \Lambda\to \susp^{s}\Lambda$ in the derived category of DG modules over $\denv \Lambda$ induces a morphism 
\[
M\cong \Lambda \lotimes{\Lambda} M \xra{\ \alpha\lotimes{\Lambda}M\ } \susp^{s} \Lambda\lotimes{\Lambda} M  \cong \susp^{s} M
\]
of complexes of $\Lambda$-modules.

\begin{lemma}
\label{le:qdiff-ann}
For any DG $\denv\Lambda$-module $X$ and integer $n$ the ideal $I=\ann_{A}\HH {n}X$ satisfies
\[
I\cdot \qdiff *A\Lambda \cdot \Ext^{n}_{\denv\Lambda}(\Lambda,X) =0\,. 
\]
In particular, $\qdiff *A{\Lambda}\cdot \Ext^{\ges c}_{\denv\Lambda}(\Lambda,X)=0$ for $c=\sup\HH *X+1$.
\end{lemma}

\begin{proof}
For each $\alpha\in\Ext^{s}_{\denv \Lambda}(\Lambda,\Lambda)$ composition gives a map of $A$-modules
\[
\Ext^{n}_{\denv\Lambda}(\Lambda,X) \xra{\ \alpha\ } \Ext^{n+s}_{\denv\Lambda}(\Lambda,X)\,.
\]
Applying $\Ext^{n}_{\denv\Lambda}(-,X)$ to \eqref{eq:factorization} and noting that $\Ext^{n}_{\denv\Lambda}(\denv\Lambda, X) = \HH nX$,  then induces a commutative diagram of graded $A$-modules
\[
\xymatrix{
\Ext^{n+s}_{\denv\Lambda}(\Lambda,X) \ar@{<-}[rr]^{\alpha}&& \Ext^{n}_{\denv\Lambda}(\Lambda,X) \\
&\HH nX  \ar@{->}[ul]^{\Ext^{n}_{\denv\Lambda}(\eta,X)} \ar@{<-}[ur]_{\Ext^{n}_{\denv\Lambda}(\mu\sfh,X)}&
}
\]
This gives the desired result.
\end{proof}

The result below is why we care about the derived Noether different.

\begin{theorem}
\label{th:qdiff-ann}
For any complexes $M,N$ of $\Lambda$-modules, and $c=\sup\Ext^{*}_{A}(M,N)+1$, one has $\qdiff *A\Lambda \cdot \Ext^{\ges c}_{\Lambda}(M,N) =0$. In particular, if $\gldim A=d<\infty$, then 
\[
\qdiff *A\Lambda\cdot \Ext^{\ges d+1}_{\Lambda}(M,N)=0
\]
for any $\Lambda$-modules $M,N$.
\end{theorem}

\begin{proof}
Apply Lemma~\ref{le:qdiff-ann} with $X=\RHom_{A}(M,N)$ and note that 
\[
\Ext^{*}_{\denv \Lambda}(\Lambda,\RHom_{A}(M,N)) \cong \Ext^{*}_{\Lambda}(M,N)\,.
\]
This last isomorphism is the derived version of the standard diagonal isomorphism in Hochschild cohomology; see, for example, \cite{AvramovIyengarLipmanNayak10}*{3.11.1}.
\end{proof}

\subsection*{Separable algebras}
Let $A$ a commutative ring. We say that $A$ is \emph{regular} if is noetherian and the  local ring $A_{\fp}$ is regular for each $\fp$ in $\Spec A$; see \cite{BrunsHerzog98}*{\S2.2}. Following Auslander and Goldman~\cite[\S1]{AuslanderGoldman60}, an $A$-algebra $\Lambda$ is \emph{separable} if it is projective as a module over $\env \Lambda$.

An $A$-algebra is \emph{Noether} if it is finitely generated as an  $A$-module. The statement below extends to the case when $A$ is only assumed to be integrally closed, at least if $\cent\Lambda$ is torsion-free as an $A$-module; see \cite{AuslanderRim63}*{Corollary 1.3(e)}.

\begin{theorem}
\label{th:separable-flat}
Let $A$ be a regular ring, $\Lambda$ a Noether $A$-algebra  that is faithful as an $A$-module, and
$\fq\in\Spec\cent\Lambda$. If $\Lambda_{\fq}$ is separable as an $A$-algebra, then it is flat, as an $A$-module.
\end{theorem}

\begin{proof}
Consider first the case when $\Lambda$ is commutative; to emphasize this we write $R$ instead of $\Lambda$. It suffices to verify that $R_{\fq}$ is flat as an $A_{\fq\cap A}$-module. It is easy to verify that $R_{\fq}$ is separable as an $A_{\fq\cap A}$-algebra, so 
the problem boils down to the following: 

Let $A$ be a regular local ring with maximal ideal $\fm$ and $R$ a Noether $A$-algebra. If $\fq$ in $\Spec R$ is such that 
$\fq\cap A=\fm$ and $R_{\fq}$ is a separable $A$-algebra, then it is flat as an $A$-module.

Since $A$ is regular, it is integrally closed, so going-down holds, and hence $\dim A\leq \dim R_{\fq}$. Other the other hand, the $A$-algebra $R_{\fq}$ is separable with $\fq\cap A=\fm$, hence $\fm R_{\fq}=\fq R_{\fq}$. This yields inequalities
\[
\dim R_{\fq}\geq \dim A = \edim A \geq \edim R_{\fq}\geq \dim R_{\fq}\,.
\]
Thus equalities hold and that implies that $R_{\fq}$ is regular as well, and any minimal generating set for $\fm$ gives a minimal generating set for $\fq R_{\fq}$, so that $R_{\fq}$ is flat as an $A$-module: compute $\Tor_{*}^{A}(A/\fm,R_{\fq})$ using a Koszul complex resolving $A/\fm$.

This completes the proof when $\Lambda$ is commutative. 

For the general case, as $\Lambda_{\fq}$ is separable as an algebra over $A$, it is separable over $A_{\fq\cap A}$. One has
\[
A_{\fq\cap A}\subseteq (\cent \Lambda)_{\fq}\cong \cent {(\Lambda_{\fq})} \subseteq \Lambda_{\fq}\,.
\]
Since $\Lambda_{\fq}$ is separable over $A_{\fq\cap A}$ so is $\cent {(\Lambda_{\fq})}$, and hence the latter is flat as an $A_{\fq\cap A}$-module, by the already established part of the result. Moreover, $\Lambda_{\fq}$ is separable (and even finitely generated) over its center, $\cent {(\Lambda_{\fq})}$,  and hence it is projective; see~\cite{AuslanderGoldman60}*{Theorem~2.1}. It follows that $\Lambda_{\fq}$ is flat as a module over $A_{\fq\cap A}$, as desired.
\end{proof}

\begin{corollary}
\label{co:separable-flat}
Assume $A$ is regular. When $\Lambda$ is a Noether $A$-algebra  that is faithful as an $A$-module, as ideals in $\cent\Lambda$ there are inclusions 
\[
\qdiff 0A\Lambda \subseteq \ndiff A\Lambda\subseteq \sqrt{\qdiff 0A\Lambda}\,.
\]
\end{corollary}

\begin{proof}
Given Lemma~\ref{le:qdiff-ndiff}, it suffices to verify that, as subsets of $\Spec \cent\Lambda$, there is an inclusion
\[
\mcV(\qdiff 0A{\Lambda})\subseteq\mcV(\ndiff A{\Lambda})\,.
\]
Pick a prime $\fq \in\Spec\cent\Lambda$ such that $\fq\not\supseteq\ndiff A\Lambda$.
Since $\Lambda_{\fq}$ is separable over $A_{\fq\cap A}$, it is flat as a module over $A_{\fq\cap A}$. Theorem~\ref{th:separable-flat} thus gives the first equality below:
\[
{\qdiff 0AR}_{\fq}\cong \qdiff 0{A_{\fq\cap A}}{R_{\fq}} = \ndiff {A_{\fq\cap A}}{R_{\fq}} \cong {\ndiff AR}_{\fq} = (\cent\Lambda)_{\fq}\,.
\]
The isomorphisms are standard and hold because localization is flat, whilst the second equality is by the hypothesis on $\fq$. It follows that $\fq\not\supseteq \qdiff 0A{\Lambda}$, as desired.
\end{proof}

Examples~\ref{ex:kdiff-qdiff} and~\ref{ex:qdiff-kdiff} show that these differents can be different.

\section{The Jacobian ideal  of a commutative ring}
\label{se:jacobian}
From this point on the focus is on commutative rings and to emphasize this we use $R$, rather than $\Lambda$,  to denote the principal ring in question. In this section we introduce a notion of a Jacobian ideal of a commutative noetherian ring and use it to prove the results announced in the Introduction. This involves yet another notion of a different of an algebra. For what follows, we have drawn often on notes of  Scheja and Storch~\cite{SchejaStorch73}. The central results are Theorems~\ref{th:jac-djac} and \ref{th:half-CM}. We write $\fitt^{R}_{d}(M)$ for the $d$th Fitting invariant of a module $M$ over a commutative ring $R$; see~\cite{BrunsHerzog98}*{pp.~21}.

\subsection*{K\"ahler different}
Let $A$ be a commutative noetherian ring and $R$ a commutative $A$-algebra.  The \emph{K\"ahler different} of $R$ over $A$ is the ideal
\[
\kdiff AR\colonequals\ \fitt^{R}_{0}(\Omega_{R/A})\,.
\]
See \cite{SchejaStorch73}*{\S15}, and also \cite{Wang94}*{Definition~4.2}, where this ideal is referred to as the Jacobian ideal of $R$ over $A$. If the ideal $\Ker(R\otimes_{A}R\to R)$ can be generated by $n$ elements, then there are inclusions
\begin{equation}
\label{eq:ndiff-kdiff}
{\ndiff AR}^{n}\subseteq \kdiff AR\subseteq \ndiff AR
\end{equation}
This is proved in, for example, \cite{SchejaStorch73}*{Satz 15.4}; see also \cite{Wang94}*{Lemma 5.8}.

\subsection*{Noether normalizations}
Let $R$ be a  noetherian ring. A \emph{Noether normalization} of $R$ is a subring $A\subseteq R$ such that the following conditions hold:
\begin{enumerate}[{\quad\rm(1)}]
\item $A$ is noetherian and of finite global dimension;
\item $R$ is finitely generated as an $A$-module.
\end{enumerate}

The following result is immediate from Corollary~\ref{co:separable-flat} and \eqref{eq:ndiff-kdiff}. In Section~\ref{se:examples} there are examples that show that the inclusions in the statement can be strict.

\begin{lemma}
\label{le:differents}
Let $A$ be a Noether normalization of a noetherian ring $R$. There are inclusions
\[
\qdiff 0AR \subseteq \ndiff AR \supseteq \kdiff AR\,,
\]
and the three ideals agree up to radical. \qed
\end{lemma}

\subsection*{Jacobian ideal}
Let $R$ be a noetherian ring, and set
\[
\jac(R) = \sum_{A\subseteq R}\kdiff AR
\]
where the summation is over all noether normalizations of $R$. We call this the \emph{Jacobian} ideal of $R$, the terminology being justified by Examples~\ref{ex:affine-algebras} and \ref{ex:complete-local}. Observe that, since $R$ is noetherian, the ideal $\jac(R)$ is  finitely generated, so only finitely many Noether normalizations are needed to compute the Jacobian ideal of $R$.

\begin{example}
\label{ex:affine-algebras}
If $R$ is an affine algebra over a field $k$, then $\jac(R)$ is the classical Jacobian ideal of the $k$-algebra $R$, namely 
\[
\jac(R) = \fitt^{R}_{d}(\Omega_{R/k}) \quad\text{where $d=\dim R$.}
\]
This is well known; see, for example, \cite[Theorem~2.3]{Wang98}.
\end{example}

\begin{example}
\label{ex:complete-local}
Let $R$ be an equicharacteristic local ring that is complete with respect to the topology defined by its maximal ideal. Then, by Cohen's Structure Theorem,  there is an isomorphism of rings
\[
R\cong \frac{k[[x_{1},\dots,x_{e}]]}{(f_{1},\dots,f_{c})}\,,
\]
where $k$ is the residue field of $R$. Let $h\colonequals\ e-\dim R$, which equals the height of the ideal $(f_{1},\dots,f_{c})$ in the ring $k[[x_{1},\dots,x_{e}]]$.  By \cite{Wang94}*{Lemma 4.3}, one has that
\[
\jac(R) = I_{h}(\{\partial f_{j}/\partial x_{i}\}_{i,j})R\,.
\]
Said otherwise, $\jac(R)$ is the ($\dim R$)th Fitting invariant of the module of continuous differentials, in the topology defined by the maximal ideal,  of $R$ over $k$.
\end{example}

\subsection*{Annihilators of Ext}
Let $R$ be a commutative ring. In what follows, we say that an ideal $I$ of $R$ \emph{annihilates $\Ext^{n}_{R}(-,-)$}, or write $I\cdot \Ext^{n}_{R}(-,-)=0$, if 
\[
I\cdot \Ext^{n}_{R}(M,N)=0 \quad\text{for all $R$ modules $M, N$.}
\]
This is equivalent to the condition that $I \cdot \Ext^{\ges n}_{R}(-,-)=0$. Note that $M$ and $N$ need not be finitely generated. In view of Examples~\ref{ex:affine-algebras} and \ref{ex:complete-local}, the result below contains Theorem~\ref{th:main1} from the Introduction.

\begin{theorem}
\label{th:jac-djac}
Let $R$ be a commutative noetherian ring of Krull dimension $d$. Then there exists an integer $s$ such that $\jac(R)^{s}$ annihilates $\Ext^{d+1}_{R}(-,-)$.
\end{theorem}

\begin{proof}
As noted before, there exist finitely many Noether normalizations, say $A_{1},\dots,A_{l}$, of $R$ such that $\jac(R)=\sum_{i}\kdiff {A_{i}}R$.
By Theorem~\ref{th:qdiff-ann},  for each $i$ the ideal $\qdiff 0{A_{i}}R$ is contained in the annihilator of $\Ext^{d+1}_{R}(-,-)$. Hence, by Lemma~\ref{le:differents}, there is an integer $n$ such that ${\kdiff {A_{i}}R}^{n}$ annihilates $\Ext^{d+1}_{R}(-,-)$. Thus the same is true of $\jac(R)^{(n-1)l+1}$.
\end{proof}

The following corollary contains the Jacobian criterion for smoothness; confer ~\cite{Matsumura89}*{Theorem~30.3} and Remark~\ref{re:singular-locus}. As will be clear from its proof, one can formulate and prove a similar statement for localizations of algebras of the type considered in Example~\ref{ex:complete-local}.

\begin{corollary}
\label{co:jac-djac}
Let $k$ be a field and $R$ an affine $k$-algebra of Krull dimension $d$. There then exists an integer $s$ such that for any localization $S$ of $R$, one has
\[
\fitt^{S}_{d}(\Omega_{S/k})^{s}\cdot \Ext^{d+1}_{S}(-,-)=0\,.
\]
Thus, if the $S$-module $\Omega_{S/k}$ is projective of rank $\leq d$, the $k$-algebra $S$ is essentially smooth.
\end{corollary}

\begin{proof}
From Theorem~\ref{th:jac-djac} and the description of Jacobian ideals of affine algebras in Example~\ref{ex:affine-algebras} it follows that, for some integer $s$, one has
\[
\fitt^{R}_{d}(\Omega_{R/k})^{s}\cdot \Ext^{d+1}_{R}(-,-) = 0 \quad\text{on $\Mod R$.}
\]
Let $U$ be a multiplicatively closed subset of $R$ such that $S\cong U^{-1}R$.  Since every $S$-module can be realized as a localization at $U$ of an $R$-module, it follows that
\[
\fitt^{R}_{d}(\Omega_{R/k})^{s}\cdot \Ext^{d+1}_{S}(-,-) = 0 \quad\text{on $\Mod S$.}
\]
It remains to note that  $\fitt^{S}_{d}(\Omega_{S/k}) = U^{-1}\fitt^{R}_{d}(\Omega_{R/k})$.

Finally assume that the $S$-module $\Omega_{S/k}$  is projective of rank $r$ with $r\leq d$. By, for example, \cite{BrunsHerzog98}*{Proposition~.14.10}, the hypothesis of projectivity translates to the equality below
\[
 \fitt^{S}_{d}(\Omega_{S/k})\supseteq  \fitt^{S}_{r}(\Omega_{S/k})=S\,,
 \]
whereas the inclusion is standard. Therefore  $\Ext^{d+1}_{S}(-,-) = 0$ on $\Mod S$, by the already established part of the result. Thus $S$ has finite global dimension. It remains to note that the hypotheses remain unchanged under extension of the ground field.
\end{proof}

\begin{remark}
\label{re:singular-locus}
In the notation of the previous corollary, $R$ is isomorphic to $k[\ulx]/I$, where $\ulx$ is a finite set of $n$ indeterminates over $k$, and $I$ is an ideal in $k[\ulx]$. Write $S\cong U^{-1}R$ for some multiplicatively closed subset $U$ in $k[\ulx]$. Then, with $h$ the height of $U^{-1}I$ in the ring $U^{-1}k[\ulx]$, the Jacobian criterion in~\cite{Matsumura89}*{Theorem~30.3} states that the $k$-algebra $S$ is smooth if the $S$-module $\Omega_{S/k}$ is projective of rank $n-h$.  Observe that $n-h\leq \dim R$, since $h$ is at least the height of $I$, and that the inequality can be strict. Thus, Corollary~\ref{co:jac-djac} offers a slight improvement on the result from \cite{Matsumura89}.
\end{remark}

Regarding Theorem~\ref{th:jac-djac}, a natural problem is to find upper bounds for the integer $s$; in particular, to understand when (not `if': see Section~\ref{se:examples}) one may take $s=1$. Theorem~\ref{th:half-CM} describes one such family of examples. Its proof requires the following result. Recall that a noetherian ring $R$ is said to be \emph{equidimensional} when the Krull dimension of $R/\fq$ remains the same as $\fq$ varies over the minimal primes $R$, and is finite.

\begin{lemma}
\label{lem:tor-independence}
Let $A\subseteq R$ be a module-finite extension of rings, where $A$ is a noetherian normal ring, $R$ is equidimensional and of finite projective dimension over $A$. 

If $ 2\, \depth R_{\fp}\ge \dim R_{\fp}$ for  each $\fp\in\Spec R$, then $\Tor^{A}_{i}(R,R)=0$ for $i\ge 1$.
\end{lemma}

\begin{proof}
Let $s=\sup \Tor^{A}_{*}(R,R)$ and pick a prime $\fq\in \Spec A$ that is minimal in the support of the $A$-module $\Tor^{A}_{s}(R,R)$. From  \cite{Auslander61}*{Theorem 1.2} one gets the equality below
\[
-s = 2\, \depth_{A_{\fq}} R_{\fq} - \depth A_{\fq} \geq 2\, \depth_{A_{\fq}} R_{\fq} - \dim A_{\fq} \,.
\]
Note that $R_{\fq}$ denotes $R\otimes_{A}A_{\fq}$. The desired result thus follows from the following claim.

\begin{claim}
Let $A\subseteq R$ be a module-finite extension of rings with $A$ local normal, and $R$ equidimensional. If $ 2\, \depth R_{\fn}\ge  \dim R_{\fn}$ for each maximal ideal $\fn$ in $R$, then $2\, \depth_{A}R \geq \dim A$.
\end{claim}

Let $\fm$ be the maximal ideal of $A$. Since $\mcV(\fm R)$ consists of the maximal ideals of $R$, it follows from \cite{BrunsHerzog98}*{Proposition~1.2.10(a)} that there exists a maximal ideal $\fn$ of $R$ such that $\depth_{R}(\fm R, R) = \depth R_{\fn}$. This justifies the first two equalities below
\[
2\, \depth_{A}R = 2\, \depth_{R}(\fm R, R) = 2\, \depth R_{\fn} \geq  \dim R_{\fn} = \dim A\,.
\]
The inequality holds by hypothesis, and the last inequality holds because $A$ is a normal domain and $R$ is a equidimensional and a module-finite extension of $A$.  The last assertion essentially follows from the going-down theorem, \cite{Matsumura89}*{Theorem~9.4}. To elaborate: Fix a minimal prime ideal $\fp$ of $R$ contained in $\fn$. We claim that $\fp\cap A=(0)$. Indeed, the induced map $A/(\fp\cap A)\subseteq R/\fp$ is also a module-finite extension, so one gets the first equality below
\[
\dim\, (A/(\fp\cap A)) = \dim\, (R/\fp) = \dim R = \dim A\,.
\]
The second one holds because $R$ is equidimensional (and this is the only place this is needed), and the last one holds because $A\subseteq R$ is module-finite. Since $A$ is a domain, we conclude that $\fp\cap A=(0)$, as claimed.

Now consider the  module-finite ring extension $A\subseteq R/\fp$ is of domains, with $A$ a normal local. Set $\fm=\fn\cap A$; this is the maximal ideal of $A$. Now we apply \cite{Matsumura89}*{Exercises 9.8 \& 9.9}, which need the normality of $A$, to deduce that the height of the ideal $\fn/\fp$ in the ring $R/\fp$ is equal to the height of $\fm$, that is to say, to $\dim A$. Thus one gets that
\[
\dim R_{\fn}\geq \height\, (\fn/\fp) = \dim A = \dim R\geq \dim R_{\fn}\,.
\]
This justifies the stated equality.
\end{proof}

The result below justifies Theorem~\ref{th:main2},

\begin{theorem}
\label{th:half-CM}
Let $R$ be a commutative noetherian ring of Krull dimension $d$. If $R$ is equidimensional and  each $\fp\in\Spec R$ satisfies $ 2\, \depth R_{\fp}\ge \dim R_{\fp}$, then  
\[
\jac(R)\cdot \Ext^{d+1}_{R}(-,-)=0\,.
\]
\end{theorem}

\begin{proof}
Lemma~\ref{lem:tor-independence} yields $\Tor^{A}_{i}(R,R) =0$ for $i\ge 1$ and hence, by Lemma~\ref{le:qdiff-ndiff}, there is an equality $\ndiff AR = \qdiff 0AR$. Theorem~\ref{th:qdiff-ann} now yields the desired inclusion.
\end{proof}

The following special case of Theorem~\ref{th:half-CM} seems worth recording.

\begin{corollary}
If a local ring $R$ is equidimensional, locally Cohen-Macaulay on the punctured spectrum, and $2\,\depth R\geq \dim R$, then $\jac(R)$ annihilates $\Ext^{d+1}_{R}(-,-)$ for $d=\dim R$.
\end{corollary}

The hypotheses of the preceding result seem close to optimal. Indeed, in Example~\ref{ex:isolated} we describe a reduced isolated singularity $R$, with $\depth R = 1$ and $\dim R=2$, for which the conclusion of the corollary does not hold. The ring $R$ is not equidimensional!

\section{The singularity category}
In this section we reinterpret results from Sections~\ref{se:de-ndiff} and ~\ref{se:jacobian} in terms of annihilators of singularity categories.
Again with an eye towards later applications, we revert to the more general setting of general (meaning, not necessarily commutative) Noether algebras. 

Let $A$ be a noetherian commutative ring and $\Lambda$ a Noether $A$-algebra. We write $\dbcat(\mod \Lambda)$ for the bounded derived category of $\mod \Lambda$. The \emph{singularity category}, also known as the \emph{stable derived category}, of $\Lambda$ is the Verdier quotient
\[
\dsing(\Lambda)\colonequals\ \dbcat(\mod \Lambda)/\thick(\Lambda)\,,
\]
where $\thick(\Lambda)$ is the subcategory of perfect complexes; see \cite{Buchweitz87}. The singularity category inherits a structure of a triangulated category from $\dbcat(\mod\Lambda)$, with suspension the usual shift functor, $\susp$, on complexes. In what follows, the morphisms between complexes $M,N$ in $\dsing(\Lambda)$ is denoted $\Hom_{\dsing}(M,N)$.

The action of the derived Hochschild cohomology algebra  on $\dbcat(\mod\Lambda)$, described in Section~\ref{se:de-ndiff}, induces an action on $\dsing(\Lambda)$. We say that an element $\alpha$ in $\Ext^{*}_{\denv A}(A,A)$ annihilates $\dsing(\Lambda)$ if for all complexes $M,N$ of $\Lambda$-modules and $n\in\bbZ$ one has
\[
\alpha\cdot \Hom_{\dsing}(M,\susp^{n}N) = 0\,. 
\]
Equivalently, the image of the morphism $M\to \susp^{s}M$, where $s=|\alpha|$, is zero in $\dsing(\Lambda)$.

\begin{proposition}
\label{pr:dsing-annihilator}
When $A$ is regular, the ideal $\qdiff *A{\Lambda}$ annihilates $\dsing(\Lambda)$.
\end{proposition}

\begin{proof}
The argument is akin to that for Lemma~\ref{le:qdiff-ann}, but a tad simpler so bears repeating. Fix an element $\alpha$ in $\qdiff sA\Lambda$ and a complex $M$ in $\dbcat(\mod \Lambda)$. Applying $-\lotimes{A}M$ to \eqref{eq:factorization} yields a commutative diagram 
\[
\xymatrix{
M\ar@{->}[rr]^{\alpha}\ar@{-->}[dr]_{\eta}&& \susp^{s} M \\
&\susp^{s}\Lambda\lotimes{A} M \ar@{->}[ur]_{\mu\sfh}&
}
\] 
of morphisms in $\dbcat(\mod\Lambda)$. It remains to observe that viewed as a complex of $A$-modules, $M$ is in $\dbcat(\mod A)$ and hence perfect; the latter conclusion holds because the global dimension of $A$ is finite. Thus, the complex of $\Lambda$-modules $\Lambda\lotimes{A} M$ is perfect, and hence zero as an object of $\dsing(\Lambda)$. It follows that $\alpha$ annihilates $\dsing(\Lambda)$.
\end{proof}

\begin{theorem}
Let $R$ be a commutative noetherian ring. Then $\jac(R)^{s} \cdot \dsing(R) = 0$ for some integer $s$; moreover, 
$s=1$ suffices if $R$ is equidimensional and \ $2\,\depth R_{\fp}\ge \dim R_{\fp}\ $ for each $\fp$ in $\Spec R$. \qed
\end{theorem}

\begin{remark}
When $k$ is a field and $\Lambda$ is a finite dimensional $k$-algebra, Proposition~\ref{pr:dsing-annihilator} identifies an ideal of $\HCH *k{\Lambda}$, the usual Hochschild cohomology algebra of $\Lambda$ over $k$, the annihilates $\dsing(\Lambda)$.  One can go a bit further,  at least for commutative rings. 

Namely, suppose  $k$ is a field and $R$ is a commutative $k$-algebra. Let $A\subseteq R$ be a Noether normalization, with $A$ a $k$-algebra. One has then a canonical morphism of graded $k$-algebras $\Ext^{*}_{\denv R}(R,R)\to \HCH *kR$, and we consider the ideal generated by the image of $\qdiff *AR$ under this map. Taking the sum over all such Noether normalizations $A$ yields an ideal in  $\HCH *kR$ that annihilates $\dsing (R)$. It seems  worthwhile to investigate this ideal. 
\end{remark}

\section{Examples}
\label{se:examples}
In this section we collect some examples that complement the results in Section~\ref{se:jacobian}. The first one illustrates that the Jacobian ideal of a commutative ring $R$ of Krull dimension $d$ need not annihilate $\Ext^{d+1}_{R}(-,-)$; confer Theorems~\ref{th:jac-djac} and \ref{th:half-CM}.

\begin{example}
\label{ex:first}
Let $k$ be a field and set $R = k[[x, y]]/(x^{5}, xy)$. Thus, $R$ is a  one-dimensional complete equicharacteristic local ring, with a unique minimal prime, namely, the ideal $(x)$. It is easily verified that the Jacobian ideal of $R$ is equal to the maximal ideal, $\fm=(x,y)$. 

We claim that $x\cdot \Ext^{2}_{R}(M,M)\ne 0$  for the $R$-module $M = R/x^{3}R$.

Indeed, the minimal free resolution of $M$ is
\[
0\lla R \xla{\ x^{3}\ } R \xla{\ [x^{2}, \ y]\ } R^{2} \xla{\bmatrix x^{3} & y & 0 \\ 0 & 0 & x \endbmatrix} R^{3}\lla \cdots
\]
Thus,  $\Ext^{2}_{R}(M,M)$ is the second cohomology of the complex
\[
0\lra M \xra{\ 0 \ } M  \xra{\bmatrix x^{2} \\  y\endbmatrix\ } M^{2} \xra{\bmatrix  0 & 0 \\ y & 0 \\ 0 & x  \endbmatrix} M^{3}\lra \cdots
\]
Evidently, the element $\xi = [x,0]$ in $M^{2}$ is a cycle. However, $x\xi$ is not a boundary element: if there exists an $f$ in $R$ such that $[x^{2},0] = [x^{2}f,yf]$ in $M^{2}$, then $x^{2}(1-f)=yf = 0$ in $M$, and this is not the case, as can be easily verified.
\end{example}

In the preceding example, the ring $R$ is equidimensional, but does not satisfy the condition on depths required to apply Theorem~\ref{th:half-CM}. In the one below, the depth condition holds, but the ring is not equidimensional.

\begin{example}
\label{ex:isolated}
Set $R =\bbC[[x,y,z]]/(xy, x(x^4-z^4))$. It is easy to check that $\jac(R)=(x,y,z^4)$. We claim the following statements hold.
\begin{enumerate}[{\quad\rm(1)}]
\item $\dim R = 2$ and $\depth R = 1$.
\item $R$ is reduced and an isolated singularity, but $R$ is not equidimensional.
\item $x\cdot \Ext^{3}_{R}(R/I,I)\ne 0$ where $I=(x^3,z)R$.  Thus  $\jac(R)$ does not annihilate $\Ext^{3}_{R}(-,-)$.
\end{enumerate}

Indeed, a primary decomposition of $(0)$ is given by
\[
(0) =  (x)\cap (x+z,y)\cap (x-z,y)\cap (x+iz,y)\cap (x-iz,y)
\]
In particular, $z$ is not a zerodivisor in $R$. The ring $S=R/zR$ is isomorphic to $k[[x,y]]/(x^5,xy)$, considered in Example~\ref{ex:first}. This will be used in the arguments. 

(1) This is clear, since $\dim S = 1$ and $\depth S=0$.

(2) Since $\sqrt{\jac(R)}=(x,y,z)$, the ring $R$ has an isolated singularity; one can check this directly or use Remark~\ref{re:singular-locus}. As $\depth R=1$ the ring $R$ satisfies Serre's condition $(S_{1})$ and so is reduced. From the primary decomposition of $(0)$ it is easy to verify that $R$ is not equidimensional; indeed $\dim R/(x)=2$ but $\dim R/\fp=1$ for any other minimal prime $\fp$.

(3) Since $R/I=S/x^3S$, it follows from  that that $x$ does not annihilate $\Ext_{S}^2(R/I,R/I)$. Consider the following presentation of the ideal $I$:
\[
0\lla I \lla R^{2}\xla{\begin{bmatrix} z & x^{2} & y \\ -x^{3} & -xz^{3} & 0 \end{bmatrix}} R^{3}
\]
Tensoring this with $R/I$ gives an exact sequence of $S$-modules
\[
0 \lla I/I^{2}\lla (R/I)^{2}\xra{\begin{bmatrix} 0 & x^{2} & y \\ 0 & 0 & 0 \end{bmatrix}} (R/I)^{3} 
\]
It follows that $I/I^{2}$ is isomorphic to the direct sum of $R/I$ and  $R/(I+(x^2,y))$. Since $I^2=zI$, the module $R/I$ is a direct summand of $I/zI$, and hence we deduce that $\Ext_S^{2}(R/I,R/I)$ is a direct summand of $\Ext_{S}^{2}(R/I,I/zI)$. It remains to note that this last module is isomorphic to $\Ext_{R}^{3}(R/I,I)$, by Rees' Theorem~\cite[Lemma~3.1.16]{BrunsHerzog98}, and hence $x\cdot \Ext_R^3(R/I,I)$ is nonzero.
\end{example}

The preceding examples can be used to show that the differents encountered in Section~\ref{se:jacobian}, namely, the K\"ahler, the Noether, and the derived Noether, can be different.

\begin{example}
\label{ex:kdiff-qdiff}
Let $R$ be as in Example~\ref{ex:first} and set $A:=k[[x+y]]$; this is a Noether normalization of $R$.  Then $y-x$ is in $\kdiff AR$ but not in $\qdiff 0AR$, so that $\kdiff AR\not\subseteq\qdiff 0AR$. 

Indeed, a direct computation yields  $\kdiff AR=(5x^{4},y-x)$. With $M$ as in Example~\ref{ex:first}, it is easy to verify (from the discussion in that example) that $y$ annihilates $\Ext^{2}_{R}(M,M)$; since $x$ does not, it follows that neither does $y-x$, and hence that  $y-x$ is not in $\qdiff 0AR$.

The K\"ahler different is always contained in the Noether different and hence, in this example, the inclusion $\qdiff 0AR\subseteq \ndiff AR$ is strict.
\end{example}

The next example, from \cite[pp.~102]{SchejaStorch73}, illustrates that the K\"ahler different can be smaller than the (derived) Noether different.

\begin{example}
\label{ex:qdiff-kdiff}
Let $R=k[x,y,z]/(x^{2}-y^{2},x^{2}-z^{2},xy,xz,yz)$, where $k$ is a field of characteristic zero. Evidently, $\qdiff 0kR=\ndiff kR$.
One can check directly that $\ndiff kR=(x^{2})$ whereas $\kdiff kR=0$.
\end{example}

\subsection*{Acknowledgements}
Part of this work was done at the American Institute of Mathematics, whilst the authors were part of a SQuaRE program. We thank the AIM
for providing a congenial atmosphere for such endeavors. In addition, SBI was partly supported by NSF grant DMS-1503044; RT was partly supported by  JSPS Grant-in-Aid for Scientific Research 25400038 and 16K05098.

\end{document}